\documentclass{amsart}
\usepackage{amsfonts, amssymb, amsmath, color, graphicx}
\usepackage{enumitem}
\usepackage{tikz}
\usetikzlibrary{arrows.meta,positioning}

\newtheorem{thm}{Theorem}
\newtheorem{cor}[thm]{Corollary}
\newtheorem{lem}[thm]{Lemma}

\theoremstyle{definition}
\newtheorem{defn}[thm]{Definition}
\newtheorem{rem}[thm]{Remark}

\def\Im{{\sf Im}\,}
\def\Re{{\sf Re}\,}

\numberwithin{thm}{section}

\begin{document}

\title{Dynamics inside the attracting basins of some skew products}

\author[J.E. Forn\ae ss, M. Hu, F. Rong]{John Erik Forn\ae ss, Mi Hu, Feng Rong}

\address{J.E. Forn\ae ss; Department of Mathematics, NTNU, Tronddheim, Norway}
\email{fornaess@gmail.com}

\address{M. Hu; School of Mathematics, Monash University, 9 Rainforest Walk, Clayton, VIC 3168, Australia}
\email{humihqu@gmail.com}

\address{F. Rong; School of Mathematical Sciences, Shanghai Jiao Tong University, 800 Dong Chuan Road, Shanghai, 200240, P.R. China}
\email{frong@sjtu.edu.cn}



\thanks{F. Rong is partially supported by the National Natural Science Foundation of China (Grant No. 12671101).}

\begin{abstract}
Polynomial skew products in $\mathbb{C}^2$ are maps of the form $F(z,w)=(P(z),Q(z,w))$, where $P$ and $Q$ are polynomials. Their local dynamics have been widely investigated. In this paper, we study the global dynamics inside Fatou components of some skew products. We consider all the inverse images in a Fatou component of a given point and use the Kobayashi metric to measure the distance between points. In the cases we consider, there are always arbitrarily large Kobayashi balls in the complement of these inverse sets.
\end{abstract}

\maketitle

\section{Introduction}

Let $\hat{\mathbb{C}}=\mathbb{C}\cup\{\infty\}$, and $f: \hat{\mathbb{C}} \rightarrow \hat{\mathbb{C}}$ be a nonconstant holomorphic map. For any $n\ge 1$, let $f^{ n}: \hat{\mathbb{C}} \rightarrow \hat{\mathbb{C}}$ be its $n$-fold iterate. In complex dynamics, two crucial disjoint invariant sets are associated with $f$, the {\sl Julia set} and the {\sl Fatou set}, which partition the space $\hat{\mathbb{C}}$ (\cite{RefM, RefF}).

The Fatou set of $f$ is defined as the largest open set where the family of iterates is locally normal. In other words, for any point $z$ in the Fatou set of $f$, there exists an open neighborhood $U$ of $z$ so that the sequence of iterates of the map restricted to $U$ forms a normal family, so the iterates are well-behaved. The complement of the Fatou set is called the Julia set. 

For any $z\in \hat{\mathbb{C}}$, the set $\{z_n\}=\{z_1=f(z_0), z_2=f^2(z_0), \cdots\}$ is called the orbit of the point $z=z_0$. The behaviour of orbits inside Fatou sets in one variable was investigated by Fornaess-Hu \cite{RefFH} and Hu \cite{RefH1, RefH2}. Hu \cite{RefH1} studied the dynamics of holomorphic polynomials on attracting basins and obtained the following result expressed with the use of the Kobayashi metric \cite{RefK}.

\begin{thm}\label{thmA}\cite{RefH1}
Suppose $f(z)$ is a polynomial of degree $N\geq 2$ on $\mathbb{C}$, $p$ is an attracting fixed point of $f(z),$ $\Omega_1$ is the immediate basin of attraction of $p$, $\{f^{-1}(p)\}\cap \Omega_1\neq\{p\}$, $\mathcal{A}(p)$ is the basin of attraction of $p$, $\Omega_i\ (i=1, 2, \cdots)$ are the connected components of $\mathcal{A}(p)$. Then, there is a constant $\tilde{C}$ so that for every point $z_0$ inside any $\Omega_i$, there exists a point $q\in \cup_k f^{-k}(p)$ inside $\Omega_i$ such that $d_{\Omega_i}(z_0, q)\leq \tilde{C}$, where $d_{\Omega_i}$ is the Kobayashi distance on $\Omega_i.$  
\end{thm} 

Theorem \ref{thmA} shows that in an attracting basin of a complex polynomial, the backward orbit of the attracting fixed point either is the point itself or accumulates at the boundary of all the components of the basin in such a way that all points of the basin lie within a uniformly bounded distance of the backward orbit, measured with respect to the Kobayashi metric. This is an interesting and innovative result.

On the other hand, Hu \cite{RefH2} proved that Theorem \ref{thmA} is no longer valid for parabolic basins of polynomials in one dimension. This is a more interesting and surprising result. In addition, Fornaess-Hu \cite{RefFH} investigated rational functions and a special case of a transcendental function in one dimension, which includes polynomials considered as maps on $\hat{\mathbb{C}}$. They proved that Theorem \ref{thmA} still holds for rational functions. However, Theorem \ref{thmA} does not hold for all transcendental functions, as demonstrated by counterexamples in the paper.

Hu \cite{RefH3} studied the same question in higher dimensions. The following is a generalization of Theorem \ref{thmA} to dimension two.

\begin{thm}{\label{the2}}\cite{RefH3}
Suppose $F(z, w)=(P(z), Q(w)),$ where $P(z), Q(w)$ are two polynomials of degree $m_1, m_2\geq2$ on $\mathbb{C},$ $P(0)=Q(0)=0,$ and $0< |P'(0)|, |Q'(0)|<1.$  Let $\Omega$ be the immediate attracting basin of $F(z, w)$. Then, there is a constant $C$ such that for every point $(z_0, w_0)\in \Omega$, there exists a point $(\tilde{z}, \tilde{w})\in \cup_k F^{-k}(0, 0),  k\geq0$ so that $d_\Omega\big((z_0, w_0), (\tilde{z}, \tilde{w})\big)\leq C$, where $d_\Omega$ is the Kobayashi distance on $\Omega$. 
\end{thm}

However, Theorem \ref{the2} is not valid for any of the following cases:
\begin{itemize} [itemsep=0pt]
\item[(1)] $P(z)=z^{m_1},  Q(w)=w^{m_2}$;
\item[(2)] $P(z)=z^m, 0<|Q'(0)|<1, $ i.e., $P'(0)=0;$
\item[(3)] $P(z)=z^m, Q'(0)=1, $  i.e., $P'(0)=0;$
\item[(4)] $0<|P'(0)|<1, Q'(0)=1;$
\item[(5)] $P'(0)=Q'(0)=1$. 
\end{itemize}

Holomorphic dynamics of polynomial skew products was firstly studied by Heinemann \cite{RefHe} and then continued by Jonsson \cite{RefJM}, Roeder\cite{RefR}. Lilov \cite{RefL}, Raissy\cite{RefJR} and Peters and Raissy \cite{RefPR} studied the dynamics of holomorphic skew products near an invariant fiber. The dynamics of skew products have been useful test cases for complex dynamics in dimension two, Boc Thaler \cite{RefB}. And Astorg et al. \cite{RefAB}\cite{RefABDPR} proved the existence of wondering domains for skew products in $\mathbb{P}^2.$

Suppose that $F$ is a polynomial skew product, i.e., $F(z, w)=(P(z), Q(z, w)),$ where $P(z), Q(z, w)$ are two polynomials of degree $m_1, m_2\geq2$ on $\mathbb{C}$ and $P(0)=0, Q(0,0)=0$. 
Hu \cite{RefH3} showed that Theorem \ref{the2} also fails in the following cases:
\begin{itemize} [itemsep=0pt]
\item[(1)] $P(z)=z^2, Q(z, w)=w^2+az, a\in\mathbb{C};$ 
\item[(2)] $P(z)=az+z^2, Q(z, w)=w^2+cw+bz, 0<|a|, |b|, |c|<<1$ and $ |a|>>|c|, |a|>>|b|, |c|>>|ab|$.
\end{itemize}

In this paper, we study more general polynomial skew products.  Fornaess and Hu \cite{RefFH} \cite{RefH3} studied attracting basins with opposite conclusions that there does or does not exist a constant $C$, in this paper, we explore semi-parabolic maps and find non existence of constant $C$ in all our cases. Overall, we prove that Theorem \ref{the2} does not hold for all these families of skew products (see Theorems \ref{T:C1}, \ref{T:C2}, \ref{T:C3} and \ref{T:C4}). 

Let $\Omega$ be the attracting basin of $F(z, w)=(P(z), Q(z, w))$ at $(0,0)$, and $U$ be the one-dimensional basin of $P(z)$ at $z=0$.

In section 2, we study the first family $F(z,w)=(z-z^2, \lambda w-2zw+aw^2)$, $0<\lambda\le 1$. We first show that $\Omega$ has the fibered structure $\Omega=\bigcup_{z\in U} V_z$, where $V_z$ is an open set (in $\mathbb{C}$) containing the origin $w=0$. Secondly, we project any connected component of $\Omega$ to $U\times \{w=0\}$. Then, we can reduce everything in dimension two to the one-dimensional parabolic basin $U$ on the $z$-coordinate, and prove Theorem \ref{T:C1} by the results in \cite{RefH2}. 

In sections 3-5, all $P(z)$ have attracting basins instead of parabolic basins, thus we cannot prove Theorems \ref{T:C2}, \ref{T:C3} and \ref{T:C4} by projecting to the $z$-coordinate since the Kobayashi distance is bounded in attracting basins by Theorem \ref{thmA} above. Hence, we need to explore more geometric properties and the dynamical behavior of $Q(z, w)$ in the $w$-direction.

In section 3, we study the second family $F(z,w)=(\lambda z-z^2, w-2zw-w^2)$, $0<\lambda<1$. We first show that $U\subset\partial \Omega$. Thus, we can project the immediate basin $\Omega^0$ to the $w$-coordinate, i.e., $\{z=0\}\times\mathbb{D}^*_R$, where $\mathbb{D}^*_R$ is a punctured disk with radius $R$. Then, we can prove Theorem \ref{T:C2}. To prove $U\subset\partial \Omega$, we need Lemma \ref{L:w=0} in which we show that for any $z\in U$ and $\epsilon>0$, there exist points $w', w''$ with $0<|w'|, |w''|<\epsilon$ such that the orbit of $(z,w')$ goes to $\infty$ while $(z,w'')$ is in $\Omega$ (See the orbit behavior near the origin in Figure \ref{fig1}).

In section 4, we study the third family $F(z,w)=(\lambda z-z^2, w-w^2-z^2)$, $0<\lambda<1$. We first change coordinates and find a convergent stable curve $w=\Psi(z)$. Then, we change the coordinate again so that the stable curve $w=\Psi(z)$ becomes $\tilde{w}=w-\Psi(z)=0$ which is contained in $\partial \Omega.$ Hence, we reduce the third family to the second family (see Lemma \ref{L:tilde}).

In section 5, we study the fourth family $F(z,w)=(\lambda z-z^2, w-w^2+azw^3)$, $0<\lambda<1$. The main difficulty of this case is that some Fatou components can be unbounded.

\section{The first family}

In this section, we consider the following family of skew products:
$$F(z,w)=(P(z),Q(z,w))=(z-z^2, \lambda w-2zw+aw^2),\ \ \ \ \ \ 0<\lambda\le 1.$$
Denote by $\Omega$ the attracting basin of $F$ at $(0,0)$.

To prove Theorem \ref{T:C1}, the main point is to project $F$ to the first variable and take advantage of the fact that the Kobayashi metric is distance decreasing. Then, we can use the fact that in the first variable the basin is parabolic.

\begin{lem}\label{L:projection}
$\Omega$ has the fibered structure $\Omega=\bigcup_{z\in U} V_z$, where $U$ denotes the one-dimensional parabolic basin of $P(z)$ at $z=0$ and each $V_z$ is an open set (in $\mathbb{C}$) containing the origin $w=0$.
\end{lem}
\begin{proof}
For any $(z_0,w_0)\in\Omega$ and $n\ge 1$, set $(z_n,w_n)=F^n(z_0,w_0)$.

First, consider the case when $a=0$. Set $\mu:=2/\lambda$. Then,
$$\begin{aligned}
z_n&=\prod_{j=0}^{n-1} (1-z_j)z_0,\\
w_n&=\lambda^n\prod_{j=0}^{n-1} (1-\mu z_j)w_0.
\end{aligned}$$
It is well-known that $z_n\sim 1/n$ for $n\gg 1$. Then, it is easy to estimate that $|w_n|\sim \lambda^n$ if $0<\lambda<1$, and $|w_n|\sim 1/n^2$ if $\lambda=1$. Hence, $\Omega$ has the fibered structure with each $V_z=\mathbb{C}$.

Next, consider the case when $a\neq 0$. Again, we have the estimates $z_n\sim 1/n$ for $n\gg 1$, and $|w_n|\sim \lambda^n$ if $0<\lambda<1$. When $\lambda=1$, under the blow-up $(z,w)=(z,zt)$, the blow-up map $\tilde{F}(z,t)$ takes the form
$$\begin{aligned}
z_1 & =z-z^2=z(1-z),\\
t_1 & =t(1-z+O(2)).
\end{aligned}$$
Then, we have the estimates $|t_n|\sim 1/n$ and $|w_n|\sim 1/n^2$. This shows that $\Omega_{\epsilon,\delta}:=V_{\epsilon,\delta}\times \{w=0\}\subset \Omega$ for small $\epsilon,\delta>0$, where
$$V_{\epsilon,\delta}=\{z\in \mathbb{C}:\ 0<|z|<\epsilon,\ |\arg z|<\delta\}.$$
Since $F^{-n}(\Omega_{\epsilon,\delta})\subset \Omega$ for all $n\ge 1$ and $\{w=0\}$ is invariant under $F$, we see that $U\times \{w=0\}\subset \Omega$.
\end{proof}

\begin{thm}\label{T:C1}
Let $F$ and $\Omega$ be as above. Denote by $S$ the grand orbit of $(1/2,0)$ under $F$. Then, for any connected component $\Omega_\alpha$ of $\Omega$ and an arbitrary constant $C>0$, there exists a point $(z_0, w_\alpha)\in \Omega_\alpha$ such that for any $(\tilde{z}, \tilde{w})\in S\cap \Omega_\alpha$, the Kobayashi distance $d_{\Omega_\alpha}((z_0, w_\alpha), (\tilde{z}, \tilde{w}))\ge C$.
\end{thm}
\begin{proof}
By Lemma \ref{L:projection}, we have a surjective projection $\pi:\Omega_\alpha\rightarrow U\times \{w=0\}$.

Denote by $S_P$ the grand orbit of $z=1/2$ under $P$. By \cite[Theorem A]{RefH2}, for an arbitrary constant $C>0$, there exists a point $z_0\in U$ such that for any $\tilde{z}\in S_P$, the Kobayashi distance $d_{U}(z_0,\tilde{z})\ge C$.
	 
Take any point $(z_0,w_\alpha)\in \Omega_\alpha$. Since the Kobayashi distance is non-increasing under the projection $\pi$, for any $(\tilde{z}, \tilde{w})\in S\cap \Omega_\alpha$, we get $d_{\Omega_\alpha}((z_0,w_\alpha),(\tilde{z},\tilde{w}))\ge d_{U}(z_0,\tilde{z})\ge C$.
\end{proof}

\begin{rem}
A similar argument works for $P(z)=z-z^{\nu+1}+O(z^{\nu+2})$, $\nu\ge 1$, and $z=1/2$ replaced by any point in $U$.
\end{rem}
   
\section{The second family}\label{S:second}

In this section, we consider the following family of skew products:
$$F(z,w)=(P(z),Q(z,w))=(\lambda z-z^2,w-2zw-w^2),\ \ \ \ \ \ 0<\lambda<1.$$
For simplicity, suppose that $\lambda=1/2$.

The difference between the first and second families is that, in the first family, the basin is parabolic in the first variable, whereas in the second family, the projection to the first variable is into an attractive basin where we know that there is no big gaps in the preimages. So to conclude that in this case there are big gaps, we must understand more deeply the Kobayashi metric on two dimensional domains.

Denote by $\Omega$ the attracting basin of $F$ at $(0,0)$. Let $U$ be the one-dimensional attracting basin of $z=0$ for $P(z)$ in the $z$-axis. Then, we have\\
1. $\{|z|<1/2\}\subset U$: $|z/2-z^2|\le |z|(|z|+1/2)<|z|$.\\
2. $U\subset \{|z|<3/2\}$: If $|z|\geq 3/2$ then $|z/2-z^2|\geq |z|(|z|-1/2)\geq |z|$.\\
3. $|w|<5$ for $(z,w)\in \Omega$: If $|w|\geq 5$ then $|w-2zw-w^2|=|w||1-2z-w|\ge |w|(5-1-2|z|)|\geq |w|(5-1-3)=|w|$.

Denote by $V$ the one-dimensional parabolic basin of $Q_0(w):=w-w^2$ at $w=0$ in the $w$-axis. For any $z\in U$, set $Q_z(w):=Q(z,w)$.

\begin{lem}
$V\subset \Omega$.
\end{lem}
\begin{proof}
Under the blow-up $(z,w)=(tw,w)$, the blow-up map $\tilde{F}(t,w)$ takes the form
$$\begin{aligned}
t_1 &=\frac{1}{2} t(1-w+O(w^2)),\\
w_1 &=w(1-w-2tw).
\end{aligned}$$
Clearly, we have that $|t_n|\lesssim \lambda'^n$ for any $\lambda'\in (1/2,1)$ when $n\gg 1$.
	
Now, for any $(z,w)\in \Omega$, write $(z_n,w_n)=F^n(z,w)$. Clearly, $|z_n|\sim \lambda^n$ when $n\gg 1$. To estimate $w_n$, we note that
$$\frac{1}{w_1}=\frac{1}{w}+1+2t+\cdots.$$
Using the estimate of $t_n$, we get that $w_n\sim 1/n$ when $n\gg 1$.
	
This shows that, for $\epsilon,\delta>0$ small enough, the cone $T_{\epsilon,\delta}:=\{(z,w):\ |z|<\epsilon |w|,\ w\in V_{\epsilon,\delta}\}$ is inside $\Omega$. Since $V\cap T_{\epsilon,\delta}\neq \emptyset$, it follows that $V\subset \Omega$.
\end{proof}
		
For any $\delta>0,$ define the trapezoid $R_\delta=\{w=u+iv, -2\delta\le u\le -\delta, |v|\le u/4\}$.

\begin{lem}\label{L:R}
For any $|z|\ll \delta$, $Q_z(R_\delta)$ contains $R_{\delta+3\delta^2/2}$ as a compact subset.
\end{lem}
\begin{proof}
Since $|z|\ll \delta$, it suffices to consider $Q_0$. Write $U+iV=Q_0(u+iv)$.
		
First, we estimate $Q_0$ on the four boundary curves of $R_\delta$:
$$L_1=\{-\delta+iv,\ |v| \le \delta/4\},\ \ \ \ \ \ L_2=\{-2\delta+iv,\ |v| \le \delta/2\},$$
$$L_3=\{u+iu/4;\ -2\delta \le u \le -\delta\},\ \ \ \ \ \ L_4=\{u-iu/4,\ -2\delta \le u \le -\delta\}.$$
		
On $L_1$: $Q_0(-\delta+iv)=-\delta-\delta^2+v^2+i(v+2\delta v), |v|<\delta/4$. Thus, $U=-\delta-\delta^2+v^2\in[-\delta-\delta^2, -\delta-\frac{15}{16}\delta^2].$
			
On $L_2$: $Q_0(2\delta+iv)=-2\delta-4\delta^2+v^2+i(v+4\delta v), |v|<\delta/2$. Thus, $U=-2\delta-4\delta^2+v^2\in[-2\delta-4\delta^2, -2\delta-\frac{15}{4}\delta^2].$

On $L_3$: $-2\delta \le u \le -\delta, v=u/4$. Then,
$$U=u-u^2+v^2=u-\frac{15}{16} u^2,\ \ \ \ \ \ V=v-2uv=\frac{u}{4}-\frac{u^2}{2},\ \ \ \ \ \ V-U/4=-\frac{17u^2}{64}<-\delta^2/4.$$

By symmetry, we have $V-U/4>\delta^2/4$ on the image of $L_4$.

It is also easy to see that $Q_0(3\delta/2)\in R_{\delta+3\delta^2/2}$. Thus, by the open mapping theorem, we see that $R_{\delta+3\delta^2/2}\subset Q_0(R_\delta)$. The compactness follows from the above estimates.
\end{proof}

\begin{lem}\label{L:w=0}
Let $z\in U$ and $\epsilon>0$. Then there exist points $w', w''$ with $0<|w'|, |w''|<\epsilon$ such that the orbit of $(z,w')$ goes to $\infty$ while $(z,w'')$ is in $\Omega$.
\end{lem}
\begin{proof}		
Since both $\Omega$ and the basin of attraction of $\infty$ are completely invariant under $F$, it suffices to prove the lemma for some large iterate $P^n(z)$. Hence, we assume without loss of generality that $0<|z|\ll 1$.

The existence of $(z,w'')$ is easy, just take an appropriate preimage of a point in the cone $T_{\epsilon,\delta}$.

The existence of $(z,w')$ follows from Lemma \ref{L:R}. Indeed, we can take $w'\in \cap_{n\ge 1} Q_{z_n}^{-n}(R_{\delta_n})$, where $\delta_0:=\delta$ and $\delta_n:=\delta_{n-1}+3\delta_{n-1}^2/2$, $n\ge 1$.

\end{proof}

As an immediate corollary, we have

\begin{cor}\label{C:U}
$U\subset \partial \Omega$.
\end{cor}
Figure \ref{fig1} shows the orbits near $U$ for $F=(\lambda z-z^2,w-2zw-w^2)$, $0<\lambda<1$. The orbits of some points go to $\infty$, some points go to $0$. It also illustrates that $U$ is in the boundary of $\Omega$ which is the above Corollary.
\begin{figure}[!htb]
	\centering
	\includegraphics[width=1\textwidth,height=0.35\textheight]{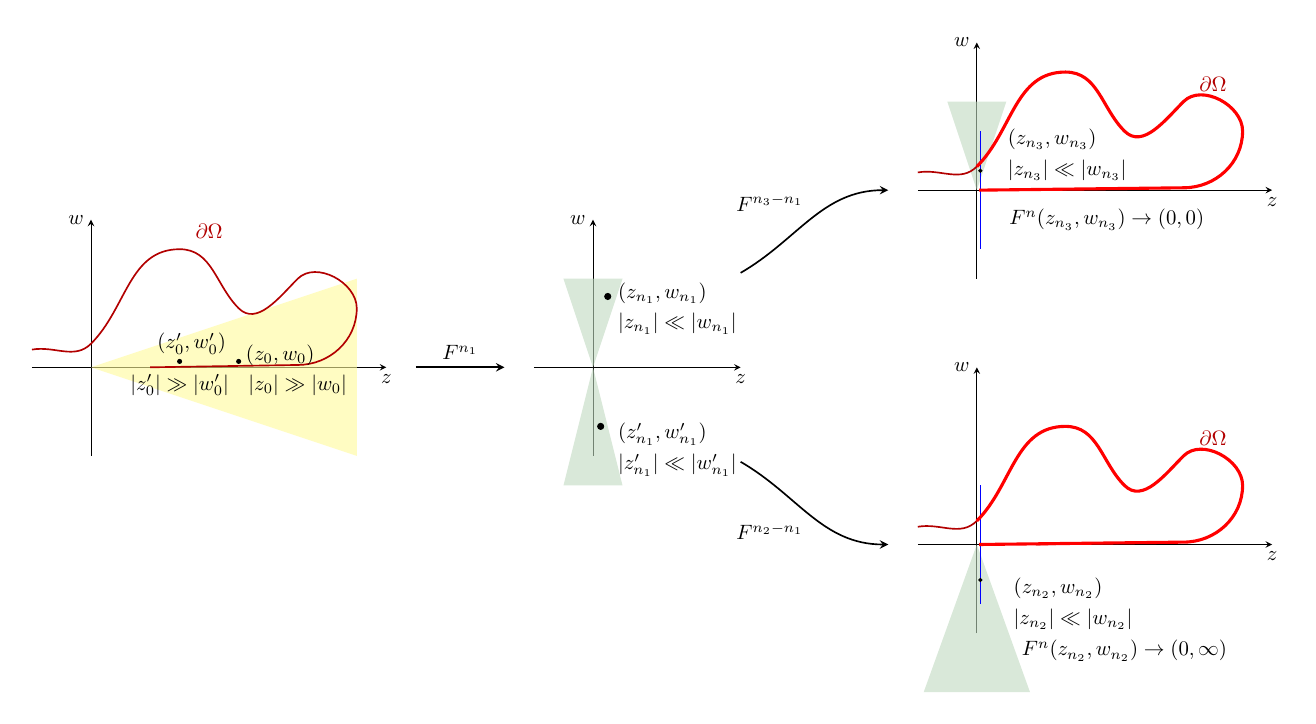}
	\caption{Orbits near the origin for $F=(\lambda z-z^2,w-2zw-w^2)$, $0<\lambda<1$.}
	\label{fig1}
\end{figure}

\begin{lem}
For any point $(z,w)\in \Omega$, its forward orbit enters the cone $T_{\epsilon,\delta}=\{(z,w):\ |z|<\epsilon |w|,\ w\in V_{\epsilon,\delta}\}$.
\end{lem}
\begin{proof}
We know that the orbit $(z_n,w_n)$ converges to zero. Using the coordinates $(Z,W)=(z,1/w)$, we have
$$W_{n+1}=\frac{W_n}{1-2Z_n-1/W_n},$$
which implies that
$$|W_{n+1}|\leq |W_n|e^{3|Z_n|}e^{2/|W_n|}.$$
Since $W_n \rightarrow \infty,$ we can assume that for large $n$, $|W_n|>L$ for any given $L$.

Hence, we get asymptotically that $|W_n|\leq K e^{2n/L}$. This implies that $W_n Z_n\rightarrow 0$, i.e., $|z_n|\leq \epsilon |w_n|$ for all large $n.$ It also implies that for all large $n$, we have $|\Im w_n|<\delta \Re w_n/4$.
\end{proof}

\begin{defn}
A connected component $\Omega^0$ of $\Omega$ is said to be an immediate basin of attraction if $F(\Omega^0)\subset \Omega^0$.
\end{defn}

\begin{cor}
There is exactly one immediate basin, $\Omega^0$.
\end{cor}

\begin{lem}\label{L:pi1}
There exists a surjective map $\pi_1:\Omega^0\rightarrow U\times \{w=0\}$.
\end{lem}
\begin{proof}
For $\delta$ small, take a curve which is a straight line in the $w$ axis from $0$ to $\delta/2.$ Then, attach a small disc $\{(z,w);\ |z|<\epsilon \delta/3,\ w=\delta/2\}.$ This configuration is contained in $\Omega^0.$ Moreove, it covers a disc $\{|z|<\epsilon\delta/3\}.$

Now take consequtive connected preimages always starting with a straight line up the real $w$ axis. These primages cover more and more of $U$ and are always contained in $\Omega^0.$
\end{proof}

\begin{thm}\label{T:C2}
Let $F(z,w)=(\lambda z-z^2,w-2zw-w^2)$, $0<\lambda<1$. Fix $\epsilon>0$ sufficiently small. Denote by $S$ the set of preimages of $(\epsilon,\epsilon)$ under $F$. For any connected component $\tilde{\Omega}$ of $\Omega$, set $\tilde{S}:=S\cap \tilde{\Omega}$. Then, for an arbitrary constant $C>0$, there exists a point $\tilde{p}\in \tilde{\Omega}$ such that for any $q\in \tilde{S}$, the Kobayashi distance $d_{\tilde{\Omega}}(\tilde{p}, q)\ge C$.
\end{thm}
\begin{proof}
First, consider $\Omega^0$, which is bounded and contains $U\times \{w=0\}$ on its boundary. Moreover, the projection $\pi_1:\Omega\rightarrow U\times \{w=0\}$ is surjective by Lemma \ref{L:pi1}.

It is easy to see that $(\epsilon,\epsilon)\in \Omega^0$ for $\epsilon>0$ sufficiently small. For any $k\ge 0$, denote by $(z^k,w^k)$ any point in $F^{-k}(\epsilon,\epsilon)\cap \Omega^0$. Since $z^k=P^{-k}(\epsilon)\rightarrow \partial U$ as $k\rightarrow \infty$, there exists $N>0$ such that for any $k\ge N$, we have $d_U(0,z^k)\ge C$. Thus, for any $w_0\in V_0$ and $k\ge N$, we have $d_{\Omega^0}((0,w_0), (z^k,w^k))\ge d_U(0,z^k)\ge C$.

Since $U\times \{w=0\}\subset \partial \Omega^0$, we have a projection $\pi_2:\Omega^0\rightarrow \{z=0\}\times \mathbb{D}^\ast_R$, where $\mathbb{D}^\ast_R$ is a punctured disk with radius $R>0$ big enough. Clearly, there exists $w_0\in V_0$ which is sufficiently close to $w=0$ such that $d_{\mathbb{D}^\ast_R}(w_0,w^k)\ge C$ for every $0\le k\le N-1$. Thus, we have $d_{\Omega^0}((0,w_0), (z^k,w^k))\ge d_{\mathbb{D}^\ast_R}(w_0,w^k)\ge C$.

For any other $\Omega_\alpha$, suppose that $F^{n_\alpha}(\Omega_\alpha)\subset \Omega^0$. Then, there exists a point $(0,w_\alpha)$ such that $d_{\Omega^0}((0,w_\alpha),F^{n_\alpha}(q))\ge C$ for any $q\in S_\alpha$. Let $p_\alpha\in \Omega_\alpha$ such that $F^{n_\alpha}(p_\alpha)=(0,w_\alpha)$. Then, we have $d_{\Omega_\alpha}(p_\alpha, q)\ge d_{\Omega^0}((0,w_\alpha),F^{n_\alpha}(q))\ge C$ for any $q\in S_\alpha$.
\end{proof}

\section{The third family}

In this section, we consider the following family of skew products:
$$F(z,w)=(P(z),Q(z,w))=(\lambda z-z^2,w-w^2-z^2),\ \ \ \ \ \ 0<\lambda<1.$$
For simplicity, suppose that $\lambda=1/2$.

A key property of the second family is that the $z$ axis stays in the boundary of $\Omega$. For the third family, we prove first that there is a local strong stable manifold near the origin. We show that this is in the boundary of $\Omega$ and then we do a local conjugation which sends this manifold to the $z$ axis at the origin. 

\begin{lem}
Let $F(z,w)=(z/2-z^2,w-w^2-z^2)$. Then there exists a complex analytic curve $w=\Psi(z)$ tangent to the $z$-axis in a small disc centered at $z=0$.
\end{lem}
\begin{proof}
Let $Z=\Phi(z),$ where \(\Phi\) is the Koenigs linearizing coordinate satisfying
$$
\Phi\!\left(\frac{z}{2}-z^2\right)=\frac{1}{2}\Phi(z),\qquad \Phi(0)=0,\qquad\Phi'(0)=1.
$$

\[
\begin{tikzpicture}[
>=Stealth,
node distance=1.4cm and 4.2cm,
every node/.style={font=\normalsize},
arrow/.style={->, thick}
]

\node (z) {$z$};
\node[right=of z] (fz) {$f=\dfrac z2-z^2$};
\node[below=of z] (Z) {$Z=\Phi(z)$};
\node[right=of Z] (LZ) {$\dfrac12 Z$};

\draw[arrow] (z) -- node[above] {$f$} (fz);
\draw[arrow] (z) -- node[left] {$\Phi$} (Z);
\draw[arrow] (fz) -- node[right] {$\Phi$} (LZ);
\draw[arrow] (Z) -- node[below] {$L$} (LZ);	
\end{tikzpicture}
\]

Then in the new coordinates $(Z,W)$, with
$$
Z=\Phi(z),\qquad W=w,
$$
we have
$$
z=\psi(Z)=\Phi^{-1}(Z).
$$
Thus, the transformed map becomes
$$
G(Z,W)=\left(\frac{Z}{2},\;W-W^2-\psi(Z)^2\right).
$$

The inverse change of coordinates $\psi(Z)$ satisfies
$$
\psi\!\left(\frac{Z}{2}\right)=\frac{1}{2}\psi(Z)-\psi(Z)^2.
$$
Therefore the transformed system is
$$
G(Z,W)=\left(\frac{Z}{2},\;W-W^2-Z^2-8Z^3-\frac{176}{3}Z^4-\cdots\right).
$$
Hence, in a small neighborhood of the origin, the map has the form
$$
G(Z,W)=\left(\frac{Z}{2},\;W-W^2+g(Z)\right),
$$
where $g(Z)=\sum_{k\geq 2} c_kZ^k$ is a convergent power series in a small disc. In particular, there exists a constant $C$ so that $|c_k|\leq C^k$ for all $k$.

The origin is a fixed point, and the Jacobin matrix at origin is
$$
DG(0,0)=
\begin{pmatrix}
1/2 & 0\\
0 & 1
\end{pmatrix}.
$$
Thus, the $Z$-direction is the strong stable direction, since its eigenvalue is $1/2$, while the $W$-direction is neutral/parabolic, since its eigenvalue is $1$.

We look for a local invariant curve tangent to the $Z$-axis. Write it as a graph of a holomorphic function $W=\phi(Z),$ with $\phi(0)=0$ and $\phi'(0)=0.$ Note that at this moment $\phi$ is a formal power series, with the coefficients defined inductively below. We will then show that this formal power series is actually convergent in a small disc.

The invariance condition means that if \((Z,\phi(Z))\) lies on the curve, then \(G(Z,\phi(Z))\) also lies on the same curve. Thus, we require
$$
G(Z,\phi(Z))=\left(\frac{Z}{2},\phi(Z)-\phi^2(Z)+g(Z)\right)=\left(\frac{Z}{2},\phi\left(\frac{Z}{2}\right)\right),
$$
i.e.,
$$
\phi\left(\frac{Z}{2}\right)=\phi(Z)-\phi^2(Z)+g(Z).
$$

Write $\phi(Z)=\sum_{k\ge 2} b_kZ^k$. Then, we require
\[\sum_{k\ge 2} \frac{b_k}{2^k} Z^{k}=\sum_{k\ge 2} b_k Z^{k}-\left(\sum_{k\ge 2} b_k Z^{k}\right)^2+\sum_{k\geq  2}c_kZ^k.\]

For each $k\geq 2$, we get
$$
b_k(1-1/2^k)=\sum_{i+j=k, i,j\geq 2}b_ib_j+c_k,
$$
i.e.,
$$
b_k=\frac{1}{1-1/2^k}\sum_{i+j=k, i,j\geq 2}b_ib_j+\frac{1}{1-1/2^k}c_k.
$$
Thus, we have
$$
|b_k|\leq 2\sum_{i+j=k, i,j\geq 1}|b_i||b_j|+2C^k.
$$
Hence,
$$
2|b_k|+2C^k \leq \sum_{i+j=k, i,j\geq 1}(2|b_i|)(2|b_j|)+4C^k,
$$
which implies that
$$
2|b_k|+2C^k\leq \sum_{i+j=k, i,j\geq 1}(2|b_i|+2C^i)(2|b_j|+2C^j).
$$

If we let $B_i=2|b_i|+2C^i$, then we get the inductive inequalities
$$B_k\leq \sum_{i+j=k, i,j \geq 1} B_i B_j.$$
Next, we replace the $B_i$ by larger numbers $B'_i$ so that
$$B'_k=\sum_{i+j=k, i,j}B'_iB'_j.$$
Note that $\{B'_k\}$ are Catalan-type numbers (see e.g. \cite{B:Catalan}).

The function $B'(t)=\sum_{k\geq 1}B'_k t^k$ satisfies a quadratic equation of the form
$$
B'(t)=B'_1t+B'(t)^2,
$$
with solutions $B'(t)=\frac{1\pm\sqrt{1-4B'_1t}}{2}$. Since $B'(0)=0$, we have $B'(t)=\frac{1-\sqrt{1-B'_1t}}{2}$. This is a convergent power series when $|t|<\frac{1}{B'_1}:=R.$ Therefore, the coefficients $B'_k$ satisfy an exponential bound $|B'_k| \le C_1^k$, for some constant $C_1>0.$

Since $|b_k|\le B'_k,$ we also get $|b_k|\le C_1^k.$ So $b_k$ grows at most exponentially. Therefore, the formal curve $W=\phi(Z)$ is actually convergent for $Z$ in a small disc around $Z=0.$ This is the stable curve of $G$ at the origin. It is invariant, tangent to the attracting $Z$-direction, and points on it converge to $(0,0)$ exponentially fast. Going back to the original coordinates, we get a convergent stable curve $w=\Psi(z)$ tangent to the $z$-axis in a small disc centered at $z=0$.
\end{proof}

Next, we connect to the analysis of the second family in section \ref{S:second} above.

We change coordinates so that the stable curve $w=\Psi(z)$ becomes the $z$-axis. We conjugate $F$ with the change of coordinates $(z,w)\rightarrow (z,\tilde{w})=(z, w-\Psi(z))$. Then, we get
$$\begin{aligned}
H(z,\tilde{w})&=(z/2-z^2, (\tilde{w}+\Psi(z))-(\tilde{w}+\Psi(z))^2-z^2-\Psi(z/2-z^2))\\
&=(z/2-z^2, \tilde{w}-\tilde{w}^2-2\tilde{w}\Psi(z)+\Psi(z)-\Psi(z)^2-z^2-\Psi(z/2-z^2)).
\end{aligned}$$
Since the curve $\tilde{w}=0$ is mapped into itself, the sum of the last four terms is zero. Therefore, we have proved the following.

\begin{lem}\label{L:tilde}
The map $F$ is conjugate to the map $H(z,\tilde{w})=(z/2-z^2, \tilde{w}-\tilde{w}^2-2\Phi(z)\tilde{w})$ near the origin.
\end{lem}

Note that this is sufficiently similar to the second family that the same analysis applies. In particular, a similar argument as in Lemma \ref{L:w=0} shows that the stable curve $\tilde{w}=0$ is contained in $\partial \Omega$.

\begin{thm}\label{T:C3}
Let $F(z,w)=(\lambda z-z^2, w-w^2-z^2)$, $0<\lambda<1$. Fix $\epsilon>0$ sufficiently small. Denote by $S$ the set of preimages of $(\epsilon,\epsilon)$ under $F$. For any connected component $\Omega_\alpha$ of $\Omega$, set $S_\alpha:=S\cap \Omega_\alpha$. Then, for an arbitrary constant $C>0$, there exists a point $p_\alpha\in \Omega_\alpha$ such that for any $q\in S_\alpha$, the Kobayashi distance $d_{\Omega_\alpha}(p_\alpha, q)\ge C$.
\end{thm}
\begin{proof}
By Lemma \ref{L:tilde}, the proof is similar to that of Theorem \ref{T:C2}. We omit the details.
\end{proof}

\section{The fourth family}

In this section, we consider the following family of skew products:
$$F(z,w)=(P(z),Q(z,w))=(\lambda z-z^2, w-w^2+azw^3),\ \ \ \ \ \ 0<\lambda<1.$$
For simplicity, suppose that $\lambda=1/2$.

For the fourth family, the difficulty is that the Fatou components can be unbounded.

Denote by $\Omega$ the attracting basin of $F$ at $(0,0)$. Let $U$ be the one-dimensional attracting basin of $z=0$ for $P(z)$ in the $z$-axis. Denote by $V$ the one-dimensional parabolic basin of $Q_0(w):=w-w^2$ at $w=0$ in the $w$-axis.

For $\delta>0$ small, set $V_\delta=\{w=u+iv,\ 0<u<\delta,\ |v|<u/100\}$ and $W_\delta=\{(z,w);\ z\in U,\ w\in V_\delta\}$.

We use the notation $A=\pm B$ to mean that $|A|\leq B$.

\begin{lem}\label{L:delta}
$F(W_\delta)\subset W_{\delta-\delta^2/2}$ whenever $\delta$ is small enough (which depends on $a$).
\end{lem}
\begin{proof}
Let $(z,w)\in U\times V_\delta.$ Pick a constant $L$ so that $|az|<L$ on $U.$ Write $w'=u'+iv'=w-w^2+azw^3.$ We then get $w'=w-w^2\pm L|w|^3$ where the last term is an upper bound for $azw^3$. Thus,
$$
u'=u-u^2+v^2\pm 2Lu^3,\ \ \ \ \ \ v'=v-2uv\pm 2Lu^3.
$$
We show first that $0<u'<\delta-\delta^2/2.$
$$
u'\geq u-u^2-2Lu^3>0,\ \ \ \ \ \ u'<u-u^2+u^2/10^4-2Lu^3<u-u^2/2<\delta-\delta^2/2.
$$
The last inequality uses that the function $u-u^2/2$ is increasing as its derivative $1-u>0.$ Next we show that $|v'|<u'/100:$
$$
v'=v(1-2u)\pm 2Lu^3,\ \ \ \ \ \ |v'|\leq |v|(1-2u)+2Lu^3<\frac{u}{100}(1-2u)+2Lu^3.
$$
We show that the right side is at most $u'/100$. This follows if we can show that
$$
u-2u^2+200Lu^3<u'.
$$
This again follows if: 
$$u-2u^2+200Lu^3< u-u^2-2Lu^3,$$
i.e.,
$$
202Lu^3<u^2,
$$
which works for small $\delta.$
\end{proof}

\begin{cor}
$W_\delta\subset \Omega$ for small $\delta>0.$
\end{cor}

\begin{lem}
For any $(z,w)\in \Omega$, its forward orbit enters $W_\delta$, $\delta$ small.
\end{lem}
\begin{proof}
Pick any $(z_0,w_0)\in \Omega.$ To show that the orbit enters any given $W_\delta$, it suffices to replace $(z_0,w_0)$ by a high iterate. Then we can assume that $(z_0,w_0)$ and all its forward orbits $(z_n,w_n)$ are close to the origin. We also must have that all $w_n$ are nonzero.

We now analyze the dynamics by using the coordinates $(Z,W)=(z,1/w)$. The function $Q(z,w)$ will take the form
$$G(Z,W)=\frac{1}{1/W-1/W^2+aZ/W^3}=W+1\pm \left[\frac{2|aZ|}{|W|}+\frac{3}{|W^2|}\right].$$

For any $\eta>0$, there exists $n_0$ such that $\|F^n(z,w)\|<\eta$ for all $n\geq n_0.$ Then, for all $n\ge n_0$, we have $|W_n|>1/\eta$, $\Re (W_{n+1}>\Re W_{n}+1/2$ and $|\Im W_{n+1}|<|\Im  W_n|+\eta.$ It follows that for all large enough $m$, we have $|\Im W_m|< (4\eta)\Re W_m.$ This means that $|\Im w_m|<4\eta \Re w_m.$ For $\eta $ small enough, we see that the orbit enters the tube $W_\delta.$
\end{proof}

For $0<\delta<\delta_0$ small enough, set $X_\delta=U\times T_\delta$, where $T_\delta=\{w=u+iv;\ -2\delta<u<-\delta,\ |v|<\delta/100\}$.

\begin{lem}\label{L:X}
$F(X_\delta)\supset X_{\delta+.99 \delta^2}$.
\end{lem}
\begin{proof}
We prove first that if $w=u+iv, u=-\delta, |v|<\delta/100$ then $u'>-\delta-.99 \delta^2$.
$$
u'=-\delta-\delta^2+\delta^2/10^4\pm2L\delta^3>-\delta-.99\delta^2.
$$

Next, we prove that if $u=-2\delta, |v|<\delta/100$ then $u'<-\delta-1.98 \delta^2$.
$$
u'=-2\delta-4\delta^2+\delta^2/10^4\pm 8L\delta^3<-2\delta-1.98 \delta^2.
$$

Next, we prove that if $-2\delta<u<-\delta, v=\delta/100$ then $v'>(\delta+.99 \delta^2)/100$ (and similar for $v=-\delta/100$).
$$\begin{aligned}
v' & =v(1-2u)\pm 2Lu^3=(\delta/100)(1-2u)\pm 2Lu^3\\
& >\delta/100+\delta^2/50-2Lu^3>(\delta+.99 \delta^2)/100.
\end{aligned}$$

Finally, one easily checks that the midpoint $w=-3/2\delta$ is mapped into $X_{\delta+.99 \delta^2}$. Thus, by the open mapping theorem, the lemma follows.
\end{proof}

\begin{lem}
$U\subset \partial \Omega$.
\end{lem}
\begin{proof}
We know that all points in $(z,0)$ are in the closure of $\Omega.$

Suppose that there is a point $z_0\in U$ so that $(z_0,0)$ is in $\Omega.$ Necessarily $z_0\neq 0.$ Then there exists a disc $D_0=\{(z_0,w);\ |w|<r\}$ inside $\Omega$. We can assume that $F^n$ converges uniformly to $(0,0)$ on $D_0.$

The image of $D_0$, $D_1$, contains a smaller disc of radius $r_1$. Inductively, $D_n:=F^n(D_0)$ contains a disc of radius $r_n.$ Note that the radius $r_n$ will shrink very slowly while $z_n$ goes fast to the origin.

After many iterates, the boundary of the disc $D_n$ intersects some $T_{\delta}$ and then we use Lemma \ref{L:X} repeatedly to show that $F^{n+k}(D_0)$ contains points of radius $\delta_0$. This contradicts with the fact that the iterates converge uniformly to $(0,0).$
\end{proof}

The following two corollaries are now immediate.

\begin{cor}
The set $V\subset \Omega$ and $\Omega\cap \{(z,w); z=0\}=V$.
\end{cor}

\begin{cor}
There is exactly one immediate basin, $\Omega^0,$ and it contains $W_\delta.$
\end{cor}

\begin{lem}\label{L:tube}
Fix $0<\delta<1/10|a|$, and let $\Omega^{0,\delta}$ be the connected component of $\Omega^0\cap \{|z|<\delta\}$ which contains $\{|z|<\delta, w\in V_{\delta}\}$. Then, $\Omega^{0,\delta} \subset \{(z,w); |z|<\delta, |w|<3\}$.
\end{lem}
\begin{proof}
Let $(z,w)$ be of the form $|z|<\delta$, $|w|=3$. Then,
$$|w_1|=|w-w^2+azw^3|\geq |w|^2-|w|-|a||z||w|^3\geq 9-3-\frac{1}{10}\times 27= 3.3>3.$$

Suppose next that $\Omega^{0,\delta}$ is not contained in the tube $\{(z,w); |z|<\delta, |w|<3\}$. Then, there is a curve $\gamma(t)=(z(t),w(t)), 0<t\leq 1$ contained in $\Omega^{0,\delta}$ which starts on the straight line in the $w$ axis from $0$ to $1/4$ and continues to a point $(z, w)\in \Omega^{0,\delta}$ with $|z|<\delta$ and $|w|>3.$ Let $t_0$ be the smallest number so that $|w(t_0)|=3.$ Let $F^n(\gamma(t))=(z_n(t),w_n(t))$. Then, there is a smallest number $t_1$ so that $|w_1(t_1)|=3$ and $t_1<t_0.$ Inductively, we get the decreasing sequence $t_n>0$ so that $|w_n(t_n)|=3$ and $t_n$ is the smallest such number. This contradicts the condition of the immediate basin, namely that the sequence $F^n([0,1])$ must converge uniformly to the origin.
\end{proof}

\begin{thm}\label{T:C4}
Let $F(z,w)=(\lambda z-z^2, w-w^2+azw^3)$, $0<\lambda<1$. Fix $\epsilon>0$ sufficiently small. Denote by $S$ the set of preimages of $(\epsilon,\epsilon)$ under $F$. For any connected component $\Omega_\alpha$ of $\Omega$, set $S_\alpha:=S\cap \Omega_\alpha$. Then, for an arbitrary constant $C>0$, there exists a point $p_\alpha\in \Omega_\alpha$ such that for any $q\in S_\alpha$, the Kobayashi distance $d_{\Omega_\alpha}(p_\alpha, q)\ge C$.
\end{thm}
\begin{proof}
First, suppose that $\Omega^0$ is bounded. Then, for $\Omega^0$, the same argument as in the proof of Theorem \ref{T:C2} works, since $\Omega^0$ is bounded and $\{w=0\}\subset \partial \Omega^0$.

Next, suppose that $\Omega^0$ is unbounded. By Lemma \ref{L:tube}, $\Omega^{0,\delta}$ is bounded and $\{w=0\}\subset \partial \Omega^{0,\delta}$. Then, we can repeat the argument above, replacing $\Omega^0$ by $\Omega^{0,\delta}$.
\end{proof}

\end{document}